\documentclass[11pt]{amsart}
\usepackage[utf8x]{inputenc}
\usepackage{amsfonts}
\usepackage{amssymb}
\usepackage{amsmath}
\usepackage{delarray}
\newtheorem{thm}{Theorem}[section]

\newtheorem{lemma}[thm]{Lemma}

\theoremstyle{definition}
\newtheorem{define}[thm]{Definition}

\numberwithin{equation}{section}

\newcommand{\cL}{\mathcal L}

\newcommand{\cX}{\mathcal X}

\newcommand{\bbR}{\mathbb R}

\newcommand{\bbC}{\mathbb C}

\begin{document}

\title{New results on the linearization of Nambu structures}

\author{Nguyen Tien Zung}
\address{Institut de Mathématiques de Toulouse, UMR5219, Université Toulouse 3}
\email{tienzung.nguyen@math.univ-toulouse.fr}

\date{Version 1, January 2012}

\maketitle

\begin{abstract} In a paper with Jean-Paul Dufour in  1999 
\cite{DufourZung-Nambu1999}, we gave a classification of linear Nambu structures, 
and obtained linearization results for Nambu structures with a nondegenerate linear part. 
There was a case left open in \cite{DufourZung-Nambu1999}, 
namely the case of smooth linearization of Nambu structures with a Type 1
hyperbolic linear part which satisfies a natural signature condition.  In this paper, we will 
show that such hyperbolic  Nambu structures are  also smoothly linearizable. 
We will also give a strong version of the analytic linearization theorem in the 
analytic case, improving a result obtained in \cite{DufourZung-Nambu1999}.
\end{abstract}

\section{Introduction}

Nambu structures were first defined by Takhtajan \cite{Takhtajan}, extending an idea of
Nambu \cite{Nambu}, as a way to generalize the Hamiltonian formalism. Since then, Nambu structures
have attracted a lot of attention from the physicists, in M-theory in particular. 
From the mathematical point of view, a Nambu structure is nothing but  an integrable 
multi-vector field in the sense that it gives rises to a singular foliation and a contravariant volume 
form on it.  We  refer to Chapter 6 of  \cite{DufourZung-PoissonBook} for an introduction to Nambu structures,
their relations with singular foliations and integrable differential forms, and the results mentioned in this section. 
We will use here a slightly different definition of Nambu structures than the original one 
given by Takhtajan \cite{Takhtajan}, in order to treat Nambu structures of order 2 on the 
same footing with Nambu structures of other orders:

\begin{define}
A Nambu structure (or tensor) of order $q$ on a manifold $M$ is a $q$-vector field $\Pi$ on $M$ which satisfies
the following integrability condition: for any point $p \in M$ such that $\Pi(p) \neq 0$, there is
a local coordinate system $(x_1,\hdots,x_n)$ in a neighborhood of $p$ such that 
$$\Pi= {\partial \over \partial x_1} \wedge \hdots \wedge {\partial \over \partial x_q}$$ 
in that neighborhood.
\end{define}

When $q \neq 2$ then the above definition is  equivalent to the original definition of Takhtajan. 
When $q = 2$ then our  definition is more restrictive: Nambu structures of order 2 in the sense of 
Takhtajan are the same as Poisson structures, while Nambu structures of order 2 in our sense are 
Poisson structure of rank (at most) 2.  A Nambu structure of order $q$ naturally gives rise to a 
singular foliation of rank $q$, whose tangent distribution near a regular point is locally generated by 
the vector fields ${\partial \over \partial x_1},\hdots,{\partial \over \partial x_q}$ in the above definition.
This foliation is equipped with a volume form written in contravariant way, which is the Nambu tensor
itself. Via the so called saturation process, any singular foliation is essentially equivalent to a foliation given by 
a Nambu structure (see Chapter 6 of \cite{DufourZung-PoissonBook}), so one can say that any foliation
can be essentially given by a Nambu structure. The differential forms which are dual to  Nambu structures are known 
under the name of integrable differential forms (htye are called co-Nambu forms in \cite{DufourZung-Nambu1999}), 
and both Nambu structures and integrable differential forms are useful in the study of singular foliations. 

Let $\Pi$ be a smooth Nambu stucture in a neighborhood of a point $O$ in an $n$-dimensional manifold 
$M$ such that $\Pi (O) = 0$.  Then the linear part $\Pi^{(1)}$ of  $\Pi$ at $O$ is a linear Nambu structure. 
According to the classification
of linear Nambu structures \cite{DufourZung-Nambu1999}, $\Pi^{(1)}$ belongs to one
of the following 2 types (in some coordinate system):

\underline{Type 1}:                                                          
$                                                                            
\Pi^{(1)} = \sum_{j=1}^{r} \pm x_j \partial / \partial x_1 \wedge ... \wedge     
\partial / \partial x_{j-1} \wedge                                         
\partial / \partial x_{j+1} \wedge ... \wedge                              
\partial / \partial x_{q+1} +   \\                                             
\sum_{j=1}^{s} \pm x_{q+1+j} \partial / \partial x_{1} \wedge ... \wedge     
\partial / \partial x_{r+j-1} \wedge                                           
\partial / \partial x_{r+j+1} \wedge                                         
\partial / \partial x_{q+1},                                                  
$                                                                            
with $0 \leq r \leq q + 1, 0 \leq s \leq \min(n-q-1, q+1-r)$.                      
                                                                             
\underline{Type 2}:                                                          
$                                                                            
\Pi^{(1)} = \partial / \partial x_{1} \wedge ... \wedge \partial / \partial x_{q-1} \wedge                                           
(\sum_{i,j=q}^{n} b^i_j x_i \partial / \partial x_j) ,                      
$                                 
where $b^i_j$ are constants.                                          

When talking about Type 1 structures, it is usually understood that $q \geq 2$, because the case $q=1$ (i.e.
vector fields) belongs to Type 2.

If $\Pi^{(1)}$ is of 
Type 2, with a nondegenerate matrix $(b^i_j)$, then $\Pi$ itself is decomposable in a neighborhood of $O$,
and the linearization problem of $\Pi$ is reduced to the very well-studied problem of linearization of a
$(n-q+1)$-dimensional vector field whose linear part is $\sum_{i,j=q}^{n} b^i_j x_i \partial / \partial x_j$. 
In particular, if the vector field $\sum_{i,j=q}^{n} b^i_j x_i \partial / \partial x_j$ is non-resonant then $\Pi$
is smoothly linearizable, and if in addition $\Pi$ is analytic and the eigenvalues of 
$\sum_{i,j=q}^{n} b^i_j x_i \partial / \partial x_j$ satisfy some diophantine conditions, then $\Pi$ is also
analytically linearizable in a neighborhood of $O$. (See Theorem 6.2 of  \cite{DufourZung-Nambu1999}). 

In this paper we will be concerned with the Type 1 case.  
Denote by 
\begin{equation}
\omega = i_{\Pi} \Omega 
\end{equation}
the integrable differential form dual to $\Pi$, 
where $\Omega = dx_1 \wedge \hdots \wedge dx_n$ is a volume form. Then  we have the following formula
for the linear part $\omega^{(1)} = i_{\Pi^{(1)}} \Omega $ of $\omega$ in the Type 1 case:
\begin{equation} \label{eqn:omega_type1}
\omega^{(1)} = dx_{q+2} \wedge ... \wedge dx_{n} \wedge dQ, 
\end{equation}
where  $\displaystyle Q =   {1 \over 2}\sum_{j=1}^{r} \pm x_j^2  + 
\sum_{j=1}^s  \pm x_{r+j} x_{q+1+j}$ is a quadratic function, 
with $0 \leq r \leq q + 1, 0 \leq s \leq \min(n-q-1, q+1-r)$.

When $r = q+1$ (then automatically $s=0$) in the above formula, we say that the linear Nambus tructure of 
type 1 is {\bf nondegenerate}.  Then the signature of $Q$ (i.e. the numbers of plus signs and minus sings), 
up to permutation, is a discrete invariant of the structure. 
A nondegenerate linear Nambu structure of type 1 is called {\bf elliptic} if the corresponding
quadratic form $Q$ is definite (positive or negative); otherwise it is called {\bf hyperbolic}. In the elliptic case,
the regular leaves of the corresponding foliation are $q$-dimensional spheres, while in the hyperbolic case
they are hyperboloids (non-compact $q$-dimensional quadrics). 

The following results are known about the linearization of Nambu structures $\Pi$ of order $q\geq 2$  
of Type 1 nondegenerate in dimension $n$ 
(see \cite{DufourZung-Nambu1999}  and Chapter 6 of \cite{DufourZung-PoissonBook}): 

1) The local singular locus $\Sigma$ of $\Pi$ (i.e. the set of points of the manifold at which $\Pi$ vanishes) near $O$
is a $(n-q-1)$-dimensional submanifold which contains $O$.

2)  $\Pi$ is formally linearizable along $\Sigma$. In other words, the exists a local smooth coordinate system
$(x_1,\hdots,x_n)$ such that $\Sigma = \{x_1 = \hdots = x_{q+1} = 0 \}$, and $\Pi = \Pi^{(1)} + F$, where
$\Pi^{(1)}$ is linear in the coordinates $(x_1,\hdots,x_n)$ and $F$ is flat at every point of $\Sigma$. (The theorem written
in \cite{DufourZung-Nambu1999} says that $\Pi$ is formally linearizable at $O$, but its proof actually shows that
it's formally linearizable along $\Sigma$, and we will need this stronger statement of formal linearization in this
paper).

3) If moreover $\Pi^{(1)}$ is elliptic, then $\Pi$ is smoothly linearizable.

4) If $\Pi$ is analytic and $\Pi^{(1)}$ is nondegenerate of Type 1 then $\Pi$ is analytically
linearizable up to multiplication by an analytic function which does not vanish at the singular point $O$.

5) There exist smooth Nambu structures whose linear part is of Type 1 hyperbolic of signature $(q-1,2)$, 
which are non-linearizable homeomorphically.

In this paper, we will complete the above results by the following linearization theorems:

\begin{thm} \label{thm:NambuHyperbolic}
Let $\Pi$ be a smooth Nambu structure of order $q \geq 3$
which vanishes at a point and whose linear part at that point is
of Type 1 nondegenerate hyperbolic with signature different from $(*,2)$ and $(2,*)$. Then $\Pi$ is locally
smothly linearizable.
\end{thm}

(The case $q=2$ is excluded from Theorem \ref{thm:NambuHyperbolic} because it is impossible: if $q=2$
and the structure is hyperbolic then its signature is  $(1,2)$ or $(2,1)$).

\begin{thm} \label{thm:NambuAnalytic}
Let $\Pi$ be an analytic (real or complex) Nambu structure which vanishes at a point and whose linear part 
at that point is nondegenerate Type 1. The $\Pi$ is locally analytically linearizable.
\end{thm}

Theorem \ref{thm:NambuAnalytic} is an  improvement of the result in \cite{DufourZung-Nambu1999} for the
analytic Type 1 case: $\Pi$ is linearizable without the need of multiplication by a function. 
Therem \ref{thm:NambuHyperbolic} was conjectured in \cite{DufourZung-Nambu1999}. When $q=n-1$
then Theorem \ref{thm:NambuHyperbolic} is a simple consequence of  Moussu's theorem \cite{Moussu2} 
about the existence of  smooth first integrals for
integrable differential 1-forms (Pfaffian systems), but when $q < n-1$ it does not follow from Moussu's
theorem, because the dual integrable differential forms will be $p$-forms with $p= n- q  > 1$, and not 1-forms.

As a direct consequence of Theorem \ref{thm:NambuHyperbolic}, we obtain the following smooth 
linearization theorem for hyperbolic integrable differential $p$-forms:

\begin{thm} Let $\omega$ be a smooth integrable differential $p$-form on a manifold $M$ and $O$
be a point of $M$ in $\omega(O) = 0$ and such that the linear part of $\omega$ at $O$ 
is nondegenerate of Type 1 hyperbolic with the signature different from $(*,2)$ and $(*,2)$. 
Then $\omega$ is smoothly  linearizable up to multiplication by a smooth function in a neighborhood of $O$. 
In other words,  there exists a smooth function $f$ and a smooth coordinate system in a neighhborhood of $O$
such that $f\omega$ is linear in that coordinate system.
\end{thm}

Recall that, unlike Nambu structures which can be linearized without the need of multiplication by a function, 
integrable differential forms are linearizable only up to multiplication by a function in general, 
because they are not closed in general.

Our proof of Theorem \ref{thm:NambuHyperbolic} is based on the smooth linearization result in the
elliptic case (!) and a surprisingly simple idea of slicing of folitations and Nambu structures .
The proof of Theorem \ref{thm:NambuAnalytic} is based on the method of Levi decomposition, 
analogous to the one used in the problem of normal forms of Poisson structures 
\cite{MonnierZung-Levi2004,Wade-Levi1997,Zung-Levi2003}. This Levi decomposition method also provides
a new simple proof of the formal linearization theorem for formal Nambu structures with a nondegenerate linear
part of Type 1.

The rest of this paper is organized as follows. Theorem \ref{thm:NambuHyperbolic} in proved in 
Section \ref{section:proof}. Theorem \ref{thm:NambuAnalytic} is proved in Section \ref{section:analytic}.
Section \ref{section:remarks} is devoted to some final remarks about related problems, results and methods.

\section{Proof of Theorem \ref{thm:NambuHyperbolic}}
\label{section:proof}

\subsection{Step 1: Slicing} Let $\Pi$ be a smooth Poisson structure of order $q$ which satisfies the hypotheses
of Theorem \ref{thm:NambuHyperbolic}. By the formal linearization result, we can assume in addition that,
, the $q$-vector field $\Pi - \Pi^{(1)}$, where $\Pi^{(1)}$ 
is the linear part of $\Pi$ in a local smooth coordinate system $(x_1, \hdots, x_n)$, is flat along
the singular locus $\Sigma = \{x_1 = \hdots = x_{q+1} = 0\}$ of $\Pi$.  Then the differential $p$-form
$\omega - \omega^{(1)}$, where $p = n-q$ and $\omega$ denotes 
the integrable differential form dual to $\Pi$ 
with respect to the volume form $d x_1 \wedge \hdots \wedge dx_n$, is also flat along $\Sigma$.
According to the hypotheses of Theorem \ref{thm:NambuHyperbolic}, we can assume that
\begin{equation} 
\omega^{(1)} = dx_{q+2} \wedge ... \wedge dx_{n} \wedge dQ, 
\end{equation}
where
\begin{equation}
Q =  {1 \over 2} (\sum_{j=1}^{k} x_j^2 - \sum_{j=k+1}^{q+1} x_j^2),
\end{equation}
with $k , q+1 -k \geq 1$ and $k, q+1-k \neq 2$.

Assume that $k \geq 3$ (if $k=1$ then this step can be skipped). Denote by 
\begin{equation}
\Lambda = (dx_{k+1} \wedge \hdots \wedge dx_{q+1})   \lrcorner \Pi
\end{equation}
the contraction of $\Pi$ with $dx_{k+1} \wedge \hdots \wedge dx_{q+1}$. Then $\Lambda$ is a 
Nambu structure of order $k-1$ whose local foliation is obtained by ``slicing'' the local 
foliation of $\Pi$ by the  spaces $\{x_{k+1} = const., \hdots, x_{q+1} = const. \}$, i.e. each leaf
of $\Lambda$ is the intersection of such a space with a leaf of $\Pi$.

Denote by $\alpha = \Lambda \lrcorner (dx_1\wedge \hdots \wedge dx_n)$ the integrable differential
form dual to $\Lambda$ with respect to the volume form  $dx_1\wedge \hdots \wedge dx_n$. The linear 
part $\alpha^{(1)}$ of $\alpha$ is:
\begin{equation}
\alpha^{(1)} = dx_{k+1}  \wedge \hdots \wedge dx_n \wedge d(\pm {1\over 2} \sum_{j=1}^k x_j^2).
\end{equation}
It means that $\Lambda$ is of Type 1 elliptic of order $k-1$, so according to the smooth linearization 
theorem for elliptic Nambu stuctures, $\Lambda$ is smoothly linearizable. 

With the the aid of a simple implicit function theorem, due to the nondegeneracy,
$\Lambda$ can also be viewed as $q+1-k$-dimensional family of elliptic Nambu structures of order $k-1$
on the spaces $\{x_{k+1} = const., \hdots, x_{q+1} = const. \}$. Applying the parametrized version
of the smooth linearization theorem for elliptic Nambu structures to this family, we obtain a local smooth
coordinate system which contains the coordinates $x_{k+1},\hdots, x_{q+1}$ and in which $\Lambda$
is linear. By renaming our new coordinate system, we can assume that 
$\Lambda = (dx_{k+1} \wedge \hdots \wedge dx_{q+1})   \lrcorner \Pi$ is already 
linear in the coordinate system $(x_1,\hdots,x_n)$.

\subsection{Step 2: $SO(k)$ symmetry group}

After Step 1, we get a local smooth coordinate system $(x_1,\hdots,x_n)$ with the following properties:

1) $\Pi$ is equal to $\Pi^{(1)}$ plus a term which is flat along the singular locus 
$\Sigma = \{x_1 = \hdots = x_{q+1} = 0\}$ of $\Pi$.

2) $\omega^{(1)} = dx_{q+2} \wedge ... \wedge dx_{n} \wedge d(\sum_{j=1}^{k} x_j^2 - \sum_{j=k+1}^{q+1} x_j^2) .$

3) $\Lambda = (dx_{k+1} \wedge \hdots \wedge dx_{q+1})   \lrcorner \Pi$ is linear.

Notice that there is a natural linear $SO(k)$ action 
(rotations in spaces $\{x_{k+1} = const. , \hdots, x_n = const. \}$) 
which preserves $\Lambda$  and $\Pi^{(1)}$ and whose orbits are exactly the leaves of the 
foliation of $\Lambda$. The leaves of $\Pi$ are invariant under this action.

Observe that, this $SO(k)$ action also preserves $\Pi$, and not only $\Lambda$  and $\Pi^{(1)}$.
Indeed, by averaging over the action of $SO(k)$, we obtain from $\Pi$ another Nambu structure $\Pi_1$ 
which has the same foliation as $\Pi$, and is $SO(k)$-invariant:
\begin{equation}
\Pi_1 := \int_{g \in SO(k)} g^*\Pi \, d\mu, 
\end{equation}
where $d\mu$ denotes the Haar measure on $SO(k)$, and $g^*\Pi$ denotes the push-forward of $\Pi$
by the action of an element $g \in SO(k)$. By construction, we also have that 
$(dx_{k+1} \wedge \hdots \wedge dx_{q+1})   \lrcorner \Pi_1 = \Lambda$, and so
$(dx_2 \wedge  \hdots \wedge dx_{q+1})   \lrcorner \Pi = (dx_2 \wedge  \hdots \wedge dx_{q+1})   
\lrcorner \Pi_1 = \pm (dx_2 \wedge \hdots \wedge dx_k) \lrcorner \Lambda = \pm x_1$. 
The form $dx_2 \wedge  \hdots \wedge dx_{q+1}$ is a volume form on the
leaves of $\Pi$ and $\Pi_1$ almost everywhere, and the above equality means that $\Pi$ coincides
with $\Pi_1$ almost everywhere (because $\Pi$ and $\Pi_1$ already have the same leaves). 
By continuity we have that $\Pi = \Pi_1$ everywhere, i.e. $\Pi$
is in fact $SO(k)$-invariant.

\subsection{Step 3: Reduction to the case of signature (1,*)} 
Starting with a coordinate system $(x_1,\hdots,x_n)$ which satisfies
the properties listed in Step 2, let us now restrict our attention to the local $(n-k+1)$-dimensional 
subspace
\begin{equation}
 V = \{x_1 = \hdots = x_{k-1} = 0\} \subset M,
\end{equation}
and consider the Nambu structure
\begin{equation}
\Theta = (dx_1 \wedge \hdots  \wedge dx_{k-1}) \lrcorner \Pi
\end{equation}
on $V$.

Notice that each non-trivial orbit of the above $SO(k)$-action 
(i.e. each leaf of  $\Lambda$, which is a $(k-1)$-dimensional sphere) intersects with $V$
at exactly two points, and there is an involution $\sigma$ on $V$ which permutes these points. Notice that
$\sigma$ preserves $\Theta$ on $V$, because $\Pi$ is $SO(k)$-invariant.

We have the following lemma:

\begin{lemma}
With the above notations, assume that  there is a smooth coordinate system $(y_k, \hdots,y_n)$ on $V$
which is equal to $(x_k,\hdots,x_n)$ plus flat terms, such that $\Theta$ is linear in this coordinate
system $(y_k, \hdots,y_n)$ and the permutation $\sigma$ is also linear in these coordinates. 
Extend the functions $y_k,\hdots,y_n$ to a neighborhood of $O$ in $M$ via the $SO(k)$-action
in Step 2 (i.e. so that they are $SO(k)$-invariant -- there is a unique way to do that). Then 
$(x_1,\hdots,x_{k-1},y_k,\hdots, y_n)$ is a local smooth coordinate system in which the Nambu structure
$\Pi$ is linear.
\end{lemma}

\begin{proof}
To say that $\sigma$ is linear in our context means that $\sigma^*y_k = - y_k$ and $\sigma^*y_j = y_j$ for
all $j \geq k+1$ on $V$. It follows that, for each $j \geq k+1$, $y_j$ is a smooth function of $n-k+1$ variables
$x_{k+1},\hdots,x_{n},r^2$ where $r^2 = x_k^2$ on $V$. Putting $r^2 = x_1^2 + \hdots + x_k^2$ instead
of $r^2=x_k^2$ in the formula for $y_j$, we get a smooth $SO(k)$-invariant extention of $y_j$ to a neighborhood
of $O$ in  $M$. Similarly, the $SO(k)$-invariant extension of $y_k$ to a neighborhood of $O$ in
$M$ is also smooth.

It remains to show that $\Pi$ is linear in  $(x_1,\hdots,x_{k-1},y_k,\hdots, y_n)$, provided that
$\Theta = (dx_1 \wedge \hdots  \wedge dx_{k-1}) \lrcorner \Pi$ is linear in  $(y_k,\hdots,y_n)$. Notice that, 
each leaf of $\Pi$ can be obtained from a leaf of $\Theta$ by taking the union of the orbits of the $SO(k)$-action
passing through it. The fact that $\Theta$ is linear implies immediately that the foliation of $\Pi$ coincides
with the foliation of a Nambu structure, i.e. $\Pi$ is linear up to multiplication by a function:
$\Pi = f \Pi^{(1)}$, where $\Pi^{(1)}$ now means the linear part of $\Pi$ in the coordinate system
$(x_1,\hdots,x_{k-1},y_k,\hdots,y_n)$, and $f$ is a function. Equality 
$(dx_1 \wedge \hdots  \wedge dx_{k-1}) \lrcorner \Pi = \Theta = \Theta^{(1)} 
= (dx_1 \wedge \hdots  \wedge dx_{k-1}) \lrcorner \Pi^{(1)}$ on $V$ implies that $f=1$ on $V$, i.e. $\Pi = \Pi^{(1)}$
at the points on $V$. But due to the $SO(k)$-invariance of both $\Pi$ and $\Pi^{(1)}$, the fact the these two
structures are equal on $V$ implies that they are equal everywhere (in a neighborhood of $O$ in $M$). Thus
$\Pi$ is linear in the coordinates $(x_1,\hdots,x_{k-1},y_k,\hdots, y_n)$. The lemma is proved.
\end{proof}

With the above lemma, the problem of linearization of $\Pi$ is now reduced to the 
problem of (equivariant with respect to $\sigma$) linearization of $\Theta$. 
If  $q+1-k > 1$ then $\Theta$ is of order $q+1-k$ and is of Type 1 hyperbolic with signature $(1, q+1-k)$.

If $q+1-k=1$ then $\Theta$ is a vector field whose linear part is of the type 
$y {\partial \over \partial y} - z {\partial \over \partial z}$. A-priori, such a linear part is degenerate in dimension
$n-k+1 > 2$, but $\Theta$ has some additional properties which will make it linreazable also in the case
$q+1-k=1$, i.e. $q=k$ and $\Theta$ is a vector field. Namely, by construction, 
$\Theta$ vanishes at the singular locus $\Sigma$ of $\Pi$,  which lies in $V$. When 
$q=k$ then $\Sigma$ is of codimension 2 in $V$, hence $\Theta$ vanishes on a submanifold of codimension 2 in
this case. Moreover,  the formal linearization of $\Pi$ along $\Sigma$ implies the formal linearization
of $\Theta$ along $\Sigma$, and the linear part of $\Theta$ is 
of the type $y {\partial \over \partial y} - z {\partial \over \partial z}$. It is these
properties which make $\Theta$ linearizable

\subsection{Step 4: Reduction to a linearization problem for vector fields}

When $q+1-k > 1$, i.e. $\Theta$ is a Nambu structure of order $q+1-k$ of Type 1 hyperbolic of signature $(1,q+1-k)$,
we can apply the above 3 steps to $\Theta$, this time to the part of negative signs in the quadratic form instead of
the part of positive signs. The problem of (equivariant with respect to $\sigma$)
linearization of $\Theta$ is then reduced to the problem of (equivariant with respect to 2 commuting involutions
$\sigma$ and $\delta$) linearization of a vector field $X$ which has properties as mentioned in the 
last paragraph of Step 3.  The following simple lemma will allow us to  linearize $X$, thus completing the proof of Theorem
\ref{thm:NambuHyperbolic}:

\begin{lemma}
 Let $X$ be a smooth vector field in a neighborhood of the origin $O$ in $\bbR^{m+2}$ with coordinates
$(z_1,\hdots, z_{m+2})$, with the following properties: \\
1) $X=0$ on the $m$-dimensional subspace $\Sigma = \{z_1 = z_2 = 0\}$. \\
2) The involutions $\sigma: (z_1, z_2,\hdots, z_{m+2}) \mapsto (-z_1, z_2,\hdots,z_{m+2})$ and
$\delta: (z_1, z_2,\hdots, z_{m+2}) \mapsto (z_1, -z_2,\hdots,z_{m+2})$ preserve $X$. \\
3)  $X = z_1 {\partial \over \partial z_1} - z_2 {\partial \over \partial z_2} + Y$, where
$Y$ is a vector field which is flat along $\Sigma$. \\
Then $X$ is locally smoothly linearizable in a $\sigma$--and--$\delta$ equivariant way.
\end{lemma}

The above lemma is in fact just a particular case of a general theorem of 
Belitskii and Kopanskii \cite{BK-Equivariant2002} (see Theorem 2.3 of that paper) about 
equivariant smooth normal forms of vector fields along central manifolds, so we will not have to
repeat its proof here. This theorem of Betliskii and Kopanskii is a generalization of
Sternberg--Chen theorem ``formal linearizability implies smooth linearizability for smooth hyperbolic
vector fields'' \cite{Chen-Vector1963,Sternberg}, and its proof uses a set of standard methods (the path
method, cohomological equations, fixed points and inverse function theorem in Banach spaces, etc.)
for dealing with smooth or $C^k$ normal forms of vector fields (see \cite{BK-Equivariant2002}).

Theorem \ref{thm:NambuHyperbolic} is proved.

\section{Formal and analytic linearization}
\label{section:analytic}

Before prove Theorem \ref{thm:NambuAnalytic} about analytic linearization, let us present here a new simple proof
of formal linearization based on Levi decomposition. 

Let $\Pi$ be a formal (complex or real) or smooth Nambu structure of order $q$ in dimension $n$
with a nondegenerate part of Type 1 at a point $O$ where $\Pi(O) = 0$. Denote by $\cL$ the set of 
(formal) vector fields which preserve $\Pi$ and tangent to the leaves of $P$. A way to obtain elements
of $\cL$ is to contract $\Pi$ with  an arbitrary closed  differential $(q-1)$-form. In particular,
by contracting $\Pi$ with diffenrential forms $dx_{i_1} \wedge \hdots \wedge dx_{i_{q-1}}$ in a coordinate
system $(x_1,\hdots,x_n)$ in which the linear part of $\Pi$ is already normalized, one sees that
the linear parts of the vector fields in $\cL$ form a simple Lie algebra $\frak{g}$, which is isomorphic to 
$so(q+1,\bbC)$ or a real form of it. According to the formal Levi decomposition theorem (see Theorem 3.1.2 of
\cite{DufourZung-PoissonBook}), $\cL$ admits a Levi factor isomorphic to $\frak{g}$. In other words,
there exists a subalgebra $\hat{\frak{g}} \subset \cL$ isomorphic to $\frak g$. Then $\hat{\frak{g}}$ can be
viewed as a formal action of $\frak{g}$ on the manifold, which can be formally linearized by Hermann's theorem
\cite{Hermann-Semisimple1968}. It implies that there is a formal coordinate system $(y_1,\hdots, y_n)$
in which the linear action of $so(q+1, \bbC)$ (or a real form of it) in question preserves our Nambu structure $\Pi$,
and whose orbits are the leaves of $\Pi$. In these coordinates, not only the folitation of $\Pi$ is linearized, but $\Pi$
is also $\frak{g}$-invariant. 
The last step in the formal linearization of $\Pi$ is to replace the coordinate system $(y_1,\hdots, y_n)$ 
by a new coordinate system of the type
$(\tilde{y}_1 := y_1 f, \hdots, \tilde{y}_{q+1} := y_{q+1}f, y_{q+2}, \hdots , y_{n})$, where
$f = f(Q,y_{q+2}, \hdots, y_n)$ is an 
appropriate formal function of $n-q$ variables, and $Q = (1/2) \sum_{j=1}^{q+1} \pm y_j^2$
is the quadratic function associated to the linear part of $\Pi$ in the coordinate system $(y_1,\hdots,y_n).$

In order to acheive analytic linearization, we will need a fast convergence algorithm, and a control on the norms,
similarly to \cite{Zung-Levi2003}: The Nambu structure will be linear up to order $2^\ell$ after Step $\ell$ of the algorithm,
i.e. after Steop $\ell$ we will have a coordinate system $(x_1^{l},\hdots, x_n^{\ell})$ in which
\begin{equation}
\Pi = \Pi^{(1)}_\ell  + o(\ell),
\end{equation}
 where $\Pi^{(1)}_\ell$ means the linear part of $\Pi$ with respect to the coordinate system $(x_1^{l},\hdots, x_n^{\ell})$ and
$o(\ell)$ means terms of degree at least $2^\ell+1$ in the Taylor expansion of $\Pi$. At step $\ell+1$, we will linearize $\Pi$ up 
to degree $2^{\ell+1}$ (i.e. the non-linear terms are of degree $\geq 2^{l+1} +1$).

Denote by $\alpha_{ij}^\ell = ({\partial \over \partial x_i^\ell} \wedge {\partial \over \partial x_j^\ell}) 
\lrcorner (d x_1^\ell \wedge \hdots \wedge d x_{q+1}^\ell)$, and $X_{ij}^\ell = \alpha_{ij}^\ell \lrcorner \Pi$,  
for $1 \leq i < j \leq {q+1}$. Then the vector fields $X_{ij}^\ell$ belong to the Lie algebra $\cL$ of (now analytic) 
vector fields which preserve $\Pi$ and tangent to the leaves of $\Pi$, and  $X_{ij}^\ell$ are linear up to degree $2^l$
in the coordinate system $(x_1^{l},\hdots, x_n^{\ell})$, and their linear parts span a Lie algebra isomorphic to
$\frak{g} = so(q+1,\bbC)$ (or a real form of it). Denote by $\cL(\ell)$ the subspace of $\cL$ consisting of vector fields or
order at least $2^\ell+1$, i.e. without terms of degree $\leq 2^\ell$. Then $\frak{g}$ has a natural linear representation on
$\cL(\ell)/\cL(\ell+1)$ by the formula
\begin{equation}
e_{ij}. X := [X_{ij},X]  \, \, \text{mod} \, \, \cL(\ell+1)
\end{equation}
where $(e_{ij})$ is the basis of $\frak{g}$.
The cochain
\begin{equation}
B_{ij,st} = [X_{ij},X_{st}] - \sum c^{uv}_{ij, st} X_{uv} \, \, \text{mod} \, \, \cL(\ell+1), 
\end{equation}
where $c^{uv}_{ij, st}$ are the structural constant of $\frak{g}$, is in fact a 2-cocycle of $\frak{g}$ with respect
to the above linear representation of $\frak{g}$ on $\cL(\ell)/\cL(\ell+1)$. Since $H^2(\frak{g}, \cL(\ell)/\cL(\ell+1)) = 0$
by Whitehead's lemma, there is a 1-cochain $Y_{ij} \in \cL(\ell)/\cL(\ell+1)$ of $\frak{g}$ whose Chevalley--Eilenberg
coboundary  is the cocycle $B_{ij,st}$. We can assume that $Y_{ij} \in \cL(\ell)$ which gives the corresponding representative
(denoted by the same symbol, by abuse of language) in $\cL(\ell)/\cL(\ell+1)$. Then by construction,
$\hat{X}_{ij} = X_{ij} -Y_{ij}$ form a Lie algebra up to degree $2^{\ell+1}$, i.e.
\begin{equation}
[\hat{X}_{ij},\hat{X}_{st}] - \sum c^{uv}_{ij, st} \hat{X}_{uv} \in \cL(\ell+1), 
\end{equation}
The non-linear terms $Z_{ij} = \hat{X}_{ij} - \hat{X}_{ij}^{(1), l} \in \cX(\ell)$ modulo $\cX(\ell+1)$, where 
$\cX(\ell)$ denotes the set of all vector fields of order $\geq 2^l +1$, form a 1-cocycle on $\cX(\ell) / \cX(\ell+1)$
with respect to a natural  linear representation of $\frak{g}$ on it. Applying the homotopy operator to this
cocycle, we obtain a vector field $Y \in \cX(\ell)$ such that $Z_{ij}$ is the coboundary of $Y$  mod $\cX(\ell+1)$.
Putting  $z_j = x_j^{\ell} - Y_j$, where $Y_j$ is the $j$-th coeffcient of $Y$ in the coordinate system
$(x_1^{\ell}, \hdots, x_n^{\ell})$, we get a new coordinate system $(z_1,\hdots,z_n)$ which 
linearizes the vector fields $\hat{X}_{ij}^l$ up to degree $2^{l+1}.$ It implies that, in these coordinates $(z_1,\hdots,z_n)$,
$\Pi$ modulo terms of degree $\geq 2^{\ell+1}+1$ is invariant with respect to the linear action of $\frak{g}$, i.e. the truncation at
degree $2^{\ell+1}$ of $\Pi$ is invariant with respect to the linear action of $\frak{g}$. By replacing $(z_1,\hdots, z_n)$ by a
(unique) appropriate coordinate system of the type $(x^{\ell+1}_1:= z_1 (1+f), \hdots, x^{\ell+1}_{q+1} := z_{q+1} (1+f), 
x^{\ell+1}_{q+2} := z_{q+2},\hdots, x^{\ell+1}_{n}: = z_n),$ where $f$ is a function which is $\frak{g}$-invariant 
(here the action of $\frak{g}$ is linear in the coordinates $(z_1,\hdots, z_n)$) and is of order at least $2^\ell+1$ , 
we can eliminates the nonlinear part of $\Pi$ (which is $\frak{g}$-invariant) up to degree $2^{\ell+1}$. In these
new coordinates $(x^{\ell+1}_{1},\hdots, x^{\ell+1}_{n})$, the Nambu structure $\Pi$ is linear up to degree $2^{\ell+1}$.

At the formal level, the above fast convergence linearization algorithm 
(eliminating nonlinear terms from degree $2^\ell+1$ to degree $2^{\ell+1}$ at Step $\ell+1$) works without any problem, 
modulo a straightforward verification of the small claims in the construction. In order to be sure that the above algorithm
also works analytically, we must control the norms of the cocycles and cochains in the process. But this control is relatively
simple, because there is no small divisor problem here. In fact, not only that the first and second cohomologies of the simple algebra
$\frak{g}$ vanish, but they also vanish in a ``normed'' way, i.e. the norm of the homotopy operators which solve the cohomological
equations are bounded. This ``bounded vanishing of cohomology'' was observed by Conn in \cite{Conn-AnalyticPoisson1984}
and  is a important feature of semisimple Lie algebras which  has been used in various 
problems of normal forms involving Lie algebras (see e.g. \cite{Conn-AnalyticPoisson1984,MMZ-Rigidity2012,
MonnierZung-Levi2004,Zung-Levi2003} and Section 3.4 of \cite{DufourZung-PoissonBook}). 
The fact that an element of $\cL(\ell)/\cL(\ell+1)$ can be lifted to an element of $\cL(\ell)$ in a norm-controlled way
is due the the nondegeneracy of $\Pi$. In order to make this statement easy to see, 
one can use the theorem of \cite{DufourZung-Nambu1999}  about analytic linearizability of $\Pi$ up to multiplication 
by a function (though it can also be done without this theorem). With this total control over the norms in the linearization process,
one can now repeats the machinary of \cite{Zung-Levi2003} (construction of a deacreasing sequence of balls whose intersection
is a neighborhood of the origin, etc.) in order to prove the convergence of the linearization process. This convergence proof 
is rather routine, so we will skeep it here, reafering the reader to \cite{Zung-Levi2003} for the details of the machinary.
Theorem \ref{thm:NambuAnalytic} is proved.

\section{Final remarks}
\label{section:remarks}

It might be possible to prove the analytic linearization theorem by using Moser's path method \cite{Moser-VolumeElement}, starting 
from the weaker theorem about linearization after multiplication. We did use the Moser's path method for the
smooth linearization of elliptic Nambu structures \cite{DufourZung-Nambu1999}, but we don't know yet how to use this
method in the analytic context. The difficulty lies at the singular locus, where the construction may cease to be analytic.
That's why we turned to Levi decomposition to find a proof.

The problem of linearization studied in this paper is related to the poroblem of the existence of local
first integrals for Nambu structures and integrable differential forms. Integrable 1-forms have been extensively
studied in the literature. In particular, the problem of existence of local first integrals for  integrable 1-forms,
was solved by Malgrange \cite{Malgrange} (in the analytic case) and Moussu \cite{Moussu,Moussu2}
(in the smooth case, using a beautiful preparation theorem of Roche \cite{Roche-Trivialization1982}). On the other hand,
there have been so far few results on general integrable $p$-forms when $p > 1$ and corresponding Nambu structures.
(See Chapter 6 of \cite{DufourZung-PoissonBook} and references therein for known results. Some particularly
nice results are due to Medeiros \cite{Medeiros-1977,Medeiros-2000}). It would be nice to have at least some results
about existence of first integrals for Nambu structures and differential forms with ``reasonable'' (but degenerate)
singularities. As was evident already from the linear case, singular points of differential $p$-forms and Nambu structures
can not be isolated in general. But there should be a natural notion similar to ``isolated singularities'' for them ?

At the foliation level, the fact that nondegenerate Type 1 Nambu structures can be linearized, may be viewed as an instance
of the phenomenon of ``singular Reeb stability'', which is similar to Reeb's stability \cite{WR} but for singular foliations.
The topological reasons behind it (the fundamental group of the leaves is trivial) is similar to the regular case of Reeb.

The problem of linearization of Nambu structure is related to the general problem of linearization of singular foliations.
Indeed, foliations given by linear Nambu structures from a very particular family (or 2 families, 
if one considers Type 1 separate from Type 2) of linear foliations. Apparently, not much is known about the linearization
of perturbations of other families of linear foliations, 
except in the cases related to linear actions of semisimple Lie algebras.

\end{document}